\documentclass{article}
\usepackage{amsmath,graphicx,mlspconfarxiv}
\usepackage{stylefile}

\newcommand{\DACCORD}[0]{{Hydra$^2$} }

\newcommand{\bT}{\mathbf{T}}

\DeclareRobustCommand{\EX}[2][{\mathbb{E}}]{\ensuremath {#1}\left[ {#2} \right]}

\newcommand{\todo}[1]{{\color{red}#1}}
\newcommand{\Exp}{\mathbb{E}}
\DeclareMathOperator{\Diag}{Diag}

\title{Fast Distributed Coordinate Descent for Non-Strongly Convex Losses\sthanks{Thanks to EPSRC for funding via grants EP/K02325X/1, EP/I017127/1 and EP/G036136/1.}}
\name{
\begin{tabular}{c}
 Olivier Fercoq \quad Zheng Qu \quad  Peter Richt\'arik\quad Martin Tak\'a\v{c} \\
\end{tabular}
}
\address{
School of Mathematics, University of Edinburgh, Edinburgh, EH9 3JZ,  United Kingdom
}

\begin{document}

\maketitle
\begin{abstract}
We propose an efficient  distributed randomized coordinate descent method for minimizing regularized non-strongly convex loss functions. The method attains the optimal $O(1/k^2)$ convergence rate, where $k$ is the iteration counter. The core of the work is the theoretical study of stepsize parameters. We have implemented the method on Archer---the largest supercomputer in the UK---and show that the method is capable of solving a (synthetic) LASSO optimization problem with 50 billion variables.
\end{abstract}

\begin{keywords}
Coordinate descent, distributed algorithms, acceleration.
\end{keywords}
\section{Introduction}
\label{sec:intro}

In this paper we are concerned with solving regularized convex optimization problems in \emph{huge dimensions} in cases when the loss being minimized is \emph{not} strongly convex. That is, we consider the problem
\begin{equation}\label{eq:MAIN1mslp}
\min_{x\in \R^d}  L(x) \eqdef f(x) + \Reg(x),
\end{equation}
where $f:\R^d\rightarrow \R$ is a convex differentiable loss function, $d$ is huge, and $\Reg(x)\equiv\sum_{i=1}^d \Reg_i(x^i)$,
where $\Reg_i: \R \rightarrow \R\cup\{+\infty\}$
are (possibly nonsmooth) convex regularizers and $x^i$ denotes the $i$-th coordinate of $x$. We make the following technical assumption about $f$: there exists a $n$-by-$d$ matrix $\bA $ such that for all $x,h\in \R^d$,
\begin{equation}\label{eq:kvadraticUpperbound}
 f(x+h) \leq f(x)+ (f'(x))^{\top} h +\tfrac{1}{2} h^{\top} \bA^\top \bA h.
\end{equation}
For examples of regularizers $R$ and loss functions $f$ satisfying the above assumptions, relevant to machine learning applications, we refer the reader to \cite{bradley2011, richtarik2012parallel,richtarik2013distributed}.

It is increasingly the case in modern applications that the data describing the problem (encoded in $\bA$ and $R$ in the above model) is so large that it does not fit into the RAM of a single computer. In such a case, unless the application at hand can tolerate slow performance due to frequent HDD reads/writes, it is often necessary to distribute the data among the nodes of a cluster and solve the problem in a distributed manner.

While efficient distributed methods exist in cases when the regularized loss $L$ is strongly convex (e.g., Hydra \cite{richtarik2013distributed}), here we assume that $L$ is \emph{not} strongly convex. Problems of this type arise frequently: for instance, in the SVM dual, $f$ is typically a non-strongly convex quadratic, $d$ is the number of examples, $n$ is the number of features, $\bA$ encodes the data, and $\Reg$ is a $0$-$\infty$ indicator function encoding box constraints (e.g., see \cite{richtarik2012parallel}). If $d > n$, then $L$ will not be strongly convex.

In this paper we propose ``Hydra$^2$'' (Hydra ``squared''; Algorithm~\ref{algo:accdistr}) -- the first distributed stochastic coordinate descent (CD) method with fast $O(1/k^2)$ convergence rate for our problem. The method can be seen both as a specialization of the APPROX algorithm \cite{fercoq2013accelerated} to the distributed setting, or as an accelerated version of the Hydra algorithm  \cite{richtarik2013distributed} (Hydra converges as $O(1/k)$ for our problem). The core of the paper forms the development of \emph{new stepsizes}, and new efficiently computable bounds on the stepsizes proposed for distributed CD methods in \cite{richtarik2013distributed}.  We show that Hydra$^2$ is able to solve a big data problem with $d$ equal to {\em 50 billion.}

%
%

\section{The Algorithm}\label{sec-algo}

Assume we have $c$ nodes/computers available. In Hydra$^2$ (Algorithm~\ref{algo:accdistr}), the coordinates $i\in [d] := \{1,2,\dots,d\}$ are first partitioned into $c$ sets $\{ {\cal P}_l,\; l=1,\dots,c\}$, each of size $|P_l| = s := d/c$. The columns of $\bA$ are partitioned accordingly, with those belonging to partition ${\cal P}_l$ stored on node $l$. During one iteration, all computers $l \in \{1,\dots,c\}$, in parallel, pick a subset $\hat{S}_l$ of $\tau$ coordinates from those they own, i.e., from ${\cal P}_l$, uniformly at random, where $1\leq \tau\leq s$ is a parameter of the method (Step~\ref{algline:generateS}). From now on we denote by  $\hat{S}$ the union of all these sets, $\hat{S}:=\cup_l \hat{S}_l$, and refer to it by the name \emph{distributed sampling}.

Hydra$^2$  maintains
two sequences of iterates: $u_k, z_k \in \R^d$. Note  that this is usually the case with accelerated/fast gradient-type algorithms~\cite{nesterov1983method, tseng2008accelerated, beck2009fista}. Also note that the output of the method is a linear combination of these two vectors. These iterates are stored and updated in a distributed way, with the $i$-th coordinate stored on computer $l$ if $i \in {\cal P}_l$.  Once computer $l$ picks $\hat{S}_l$, it computes for each $i \in \hat{S}_l$  an update scalar $t^i_k$, which is then used to update $z_k^i$ and $u_k^i$, in parallel (using the multicore capability of computer $l$).

\begin{algorithm}
\algsetup{indent=0.8em}
\algsetup{linenodelimiter=\!}
\begin{algorithmic}[1]
 \STATE \textbf{INPUT}:  $\{{\cal P}_l\}_{l=1}^c$, $1\leq \tau\leq s$, $\{\bD_{ii}\}_{i=1}^d$, $z_0\in \R^d$
\STATE set $\theta_0={\tau}/{s}$ and  $u_0=0$
\FOR{ $k \geq 0$}
\renewcommand{\algorithmicfor}{\textbf{for each}}
\renewcommand{\algorithmicdo}{\textbf{in parallel do}}
\renewcommand{\algorithmicendfor}{\textbf{end parallel for}}
\STATE  $z_{k+1} \leftarrow z_k$,\; $u_{k+1} \leftarrow u_k$
\FOR{ computer $l\in \{1,\dots,c\}$}
\STATE pick a random set of coordinates $\hat S_l \subseteq \pP_l$, $|\hat S_l|=\tau$ \label{algline:generateS}
\FOR{ $i\in  \hat S_l$}
\STATE $t_{k}^{i} = \displaystyle \arg\!\min_{t}  f_i' (\theta_k^2 u_k + z_k) t  + \tfrac{ s\theta_k  \bD_{ii}}{2\tau}t^2
+ \Reg_i(z_{k}^{i}+t)$ \label{algline:proxop} \vspace{-1ex}
\STATE $z_{k+1}^{i} \leftarrow z_{k}^{i} + t_{k}^{i}$, \; $u_{k+1}^{i} \leftarrow u_k^{i} - (\tfrac{1}{\theta_k^2}-\frac{s}{\tau \theta_k})t_{k}^{i}$ \label{algline:updateu}
\ENDFOR
\ENDFOR
\STATE $\theta_{k+1}=\tfrac{1}{2}(\sqrt{\theta_k^4+4 \theta_k^2} - \theta_k^2)$ \label{algline:theta_k}

\ENDFOR
\STATE OUTPUT: $\theta_{k}^2 u_{k+1} + z_{k+1}$
\end{algorithmic}
\caption{\DACCORD  }
\label{algo:accdistr}
\end{algorithm}

The main work is done in Step~\ref{algline:proxop} where the update scalars $t_k^i$ are computed.
This step is a  proximal forward-backward operation similar to what is to be done in FISTA~\cite{beck2009fista}
but we only compute the proximal operator for the coordinates in $\hat{S}$.

The partial derivatives are computed at $(\theta_k^2 u_k + z_k)$. For the algorithm to be efficient,
the computation should be performed \emph{without  computing the sum} $\theta_k^2 u_k + z_k$.
As shown in~\cite{fercoq2013accelerated},
this can be done for functions $f$ of the form
\begin{equation*}\textstyle{
f(x)=\sum_{j=1}^m \phi_j(e_j^T \bA x),}
\end{equation*}
where for all $j$, $\phi_j$ is a one-dimensional differentiable (convex) function.
Indeed, let $D_i$ be the set of such $j$ for which $\bA_{ji}\neq 0$. Assume we store and update $r_{u_k}=\bA u_k$ and $r_{z_k} = \bA z_k$, then
\begin{equation*}\textstyle{
  \nabla_i f(\theta_k^2 u_k+z_k) = \sum_{j \in D_i} \bA_{ji}^T \phi'_j(\theta_k^2 r_{u_k}^j + r_{z_k}^j ).}
\end{equation*}
The average cost for this computation is $\frac{\tau}{s}\sum_{i=1}^d {\cal O}(|D_i|)$, i.e., each processor
needs to access the elements of only one column of the matrix $\bA$.


Steps~\ref{algline:proxop}
and \ref{algline:updateu} depend on a deterministic scalar sequence $\theta_k$, which is being updated in Step~\ref{algline:theta_k} as in~\cite{tseng2008accelerated}. Note that by taking $\theta_k=\theta_0$ for all $k$, $u_k$ remains equal to 0, and Hydra$^2$ reduces to Hydra~\cite{richtarik2013distributed}.

The output of the algorithm is $x_{k+1}=(\theta_{k}^2 u_{k+1} + z_{k+1})$. We only need to compute
this vector sum at the end of the execution and when we want to track $L(x_k)$.
Note that one should not evaluate $x_k$ and $L(x_k)$ at each iteration
since these computations have a non-negligible cost.


%

\section{Convergence rate}\label{sec:CRA}

The magnitude of $\bD_{ii}>0$ directly influences the size of the update  $t_k^i$. In particular, note that when there is no regularizer ($R_i=0$), then
 $t_k^i = 2\tau f_i' (\theta_k^2 u_k + z_k)(s\theta_k\bD_{ii})^{-1}$. That is, small $\bD_{ii}$ leads to a larger ``step'' $t_k^i$ which is used to update $z_k^i$ and $u_k^i$.
For this reason we refer to $\{\bD_{ii}\}_{i=1}^d$ as \firstdef{stepsize parameters}.
Naturally, some technical assumptions on $\{\bD_{ii}\}_{i=1}^d$ should be made
in order to guarantee the convergence of the algorithm.  The so-called ESO (Expected Separable Overapproximation) assumption has been introduced
in this scope. For  $h\in \R^d$, denote $h^{\hat{S}} := \sum_{i \in \hat{S}} h^i e_i$, where $e_i$ is the $i$th unit coordinate vector.

\begin{assump}[ESO]\label{ass-eso}
Assume that for all  $x\in \R^d$ and $h\in\R^d$ we have
  \begin{equation}\label{eq:ESO}
 \EX{f(x+h^{\hat{S}})} \leq f(x) + \tfrac{\EX{|\hat S|}}{d}\big((f'(x))^{\top} {h} + \tfrac{1}{2} h^{\top} \bD h \big),
 \end{equation}
 where $\bD$ is a diagonal matrix with diag.\ elements $\bD_{ii}>0$ and $\hat{S}$ is the distributed sampling described above.
 \end{assump}
The above ESO  assumption involves the smooth function $f$, the sampling $\hat S$ and the parameters $\{\bD_{ii}\}_{i=1}^d$. It has been first introduced by
Richt\'{a}rik and Tak\'{a}\v{c}~\cite{richtarik2012parallel} for proposing a generic approach
in the convergence analysis of the Parallel Coordinate Descent Methods (PCDM).
Their generic approach boils down the convergence analysis of the whole class of PCDMs to the problem of finding proper parameters which make the ESO assumption hold.
The same idea has been extended to the analysis of many variants of PCDM, including the Accelerated Coordinate Descent algorithm~\cite{fercoq2013accelerated} (APPROX)
and the Distributed Coordinate Descent method~\cite{richtarik2013distributed} (Hydra).

In particular, the following complexity result, under the ESO assumption~\ref{ass-eso}, can be deduced from \cite[Theorem 3]{fercoq2013accelerated} using Markov inequality:
 \begin{thm}[\cite{fercoq2013accelerated}]\label{thm:main}
Let $x_*\in \R^d$ be any optimal point of \eqref{eq:MAIN1mslp},  $d_0\eqdef x_0-x_*$, $L_0\eqdef L(x_0)-L(x_*)$ and let \[C_1 \eqdef \left(1-\tfrac{\tau}{s}\right)L_0,\quad C_2 \eqdef \tfrac{1}{2}d_0^\top \bD d_0.\] Choose $0 <\rho < 1$ and $0<\epsilon < L_0$ and
\begin{equation}\label{eq-KC1C2}
k\geq \frac{2s}{\tau} \left(\sqrt{\frac{C_1+C_2}{\rho \epsilon}}-1\right)+1.
\end{equation}
Then under Assumption~\ref{ass-eso}, the iterates $\{x_k\}_{k\geq 1}$ of Algorithm~\ref{algo:accdistr} satisfy
$\Prob(L(x_k)-L(x_*)\leq \epsilon) \geq 1-\rho \enspace.$
\end{thm}
The bound on the number of iterations $k$ in~\eqref{eq-KC1C2} suggests to choose the smallest possible $\{\bD_{ii}\}_{i=1}^d$ satisfying the ESO assumption~\ref{ass-eso}.

\section{New Stepsizes}
In this section we propose new stepsize parameters $\bD_{ii}$ for which
the ESO assumption is satisfied.

For any matrix $\bG\in \R^{d\times d}$, denote by  $D^{\bG}$ the diagonal matrix such that $D^\bG_{ii}= \bG_{ii}$ for all $i$
 and $D^\bG_{ij}=0$ for $i \neq j$.
We denote by $B^{\bG}\in \R^{d\times d}$  the block matrix
associated with the partition $\{\pP_1,\dots,\pP_c\}$ such that $B^{\bG}_{ij} = \bG_{ij}$ whenever $\{i,j\}\subset \pP_l$ for some $l$, and $B^\bG_{ij} = 0$ otherwise.

Let $\omega_j$ be the number of nonzeros
 in the $j$th row of $\bA$ and $\omega'_j$ be the number of ``partitions active at row $j$'', i.e.,
the number of indexes $l \in \{1,\dots,c\}$ for which the set $\{i \in {\cal P}_l : \bA_{ji} \neq 0\}$ is nonempty.
Note that since $\bA$ does not have an empty row or column, we know that $1\leq \omega_j\leq n$ and $1\leq \omega_j'\leq c$. Moreover,
 we have the following characterization of these two quantities. For notational convenience we denote $\bM_j = \bA_{j:}^\top \bA_{j:}$ for all
$j\in\{1,\dots,n\}$.
\begin{lemma}\label{l-dso} For $j\in\{1,\dots,n\}$ we have:
\begin{align}
\label{eq:omega_j}\omega_j= \max \{x^\top \bM_j x \;:\; x^\top D^{\bM_j} x \leq 1\},\\
\label{eq:omegaPRIME}\omega_j' = \max \{x^\top \bM_j x \;:\; x^\top B^{\bM_j} x \leq 1\}.
\end{align}
\end{lemma}
\begin{proof}
For $l \in \{1,\dots,c\}$ and $y \in \R^d$, denote $y^{(l)}:=(y^i)_{i \in {\cal P}_l}$. That is, $y^{(l)}$ is the subvector of $y$ composed of coordinates  belonging to partition ${\cal P}_l$.
Then we have:
$$ 
x^\top \bM_j x=(\sum_{l=1}^c \bA_{j:}^{(l)} x^{(l)})^2,\enspace x^\top B^{\bM_j}x=\sum_{l=1}^c ( \bA_{j:}^{(l)} x^{(l)})^2.
$$
Let $S' = \{l: \bA_{j:}^{(l)} \neq 0\}$, then $\omega_j'=|S'|$ and by Cauchy-Schwarz we have
\[\textstyle{ (\sum_{l\in S'}\bA_{j:}^{(l)} x^{(l)})^2 \;\;\leq\;\; \omega_j' \sum_{l\in S'} (\bA_{j:}^{(l)} x^{(l)})^2. }\]
Equality is reached when  $\bA_{j:}^{(l)} x^{(l)}=\alpha$ for some constant $\alpha$ for all $l\in S'$ (this is feasible since the subsets $\{\pP_1,\dots,\pP_c\}$ are disjoint).  Hence, we proved~\eqref{eq:omegaPRIME}.
The characterization~\eqref{eq:omega_j} follows from \eqref{eq:omegaPRIME} by setting $c=d$.
\end{proof}

We shall also need the following lemma.
\begin{lemma}[\cite{richtarik2013distributed}]\label{l-xSGxS} Fix arbitrary $\bG \in \R^{d\times d}$ and $x\in \R^d$ and let $s_1=\max\{1,s-1\}$. Then $\Exp[(x^{\hat{S}})^\top \bG x^{\hat{S}}]$ is equal to
\begin{equation}\label{eq:ESO-identity} \tfrac{\tau}{s} \left[ \alpha_1 x^\top D^\bG x + \alpha_2 x^{\top} \bG x + \alpha_3 x^\top (\bG - B^{\bG})x   \right],\end{equation}
where $\alpha_1 = 1-\tfrac{\tau-1}{s_1}$, $\alpha_2 = \tfrac{\tau-1}{s_1}$, $\alpha_3=\tfrac{\tau}{s}-\tfrac{\tau-1}{s_1}$.
\end{lemma}
Below is the main result of this section.
\begin{thm}
\label{prop:eso}
For a convex differentiable function $f$ satisfying~\eqref{eq:kvadraticUpperbound} and a distributed sampling $\hat S$ described in Section~\ref{sec-algo}, the ESO assumption~\ref{ass-eso} is satisfied for
\begin{align}\label{a-bDiibetaj}
\bD^1_{ii} \equiv \bD_{ii}=  \sum_{j=1}^n \alpha^*_j\bA_{ji}^2, \quad i = 1,\dots,d \enspace,
\end{align}
where
\begin{equation}\label{eq-alphaj}
\begin{split}
&\alpha^*_{j}:=\alpha^*_{1,j}+\alpha^*_{2,j},\enspace
\alpha^*_{1,j} = 1+\tfrac{(\tau-1)(\omega_j-1)}{s_1},\\ & \alpha^*_{2,j} =(\tfrac{\tau}{s} - \tfrac{\tau-1}{s_1}) \tfrac{\omega_j'-1}{\omega_j'}\omega_j.\\
\end{split}
\end{equation}
\end{thm}
\begin{proof}
Fix any $h\in \R^ d$. It is easy to see that
$
\E[(f'(x))^{\top}h^{\hat{S}}]=\frac{\E[|\hat S|]}{d}(f'(x))^{\top}h)
$; this follows by noting that $\hat S$ is a uniform sampling (we refer the reader to~\cite{richtarik2012parallel} and~\cite{fercoq2013smooth} for more identities of this type). In view of  \eqref{eq:kvadraticUpperbound},  we only need to show that
\begin{align}\label{a-EhAh0}
\E\left[(h^{\hat{S}})^\top \bA^\top \bA h^{\hat{S}}\right]\leq \tfrac{\E[|\hat S|]}{d} h^\top \bD h\enspace.
\end{align}
By applying Lemma~\ref{l-xSGxS} with $\bG=\bM_j$ and using Lemma~\ref{l-dso}, we can upper-bound $\E[(h^{\hat{S}})^\top \bM_j h^{\hat{S}}]$ by
\begin{align}
&\leq
\tfrac{\tau}{s}  [ \alpha_1 h^\top D^{\bM_j} h + \alpha_2 \omega_j h^{\top} D^{\bM_j} h + \alpha_3 (1-\tfrac{1}{\omega_j'})h^\top \bM_j h   ]
\notag\\
&\leq \tfrac{\tau}{s}  [ \alpha_1  + \alpha_2 \omega_j + \alpha_3 (1-\tfrac{1}{\omega_j'}) \omega_j]h^\top D^{\bM_j} h \notag\\
&= \tfrac{\E[|\hat S|]}{d} \alpha_j^* h^\top D^{\bM_j}h.\label{a-Mj}
\end{align}
Since $
\E [(h^{\hat{S}})^\top \bA^\top \bA h^{\hat{S}} ]
= \sum_{j=1}^n \E [(h^{\hat{S}})^{\top} \bM_j h^{\hat{S}}]
$,
 \eqref{a-EhAh0} can be obtained by summing up~\eqref{a-Mj} over
$j$ from 1 to $n$.
\end{proof}

\section{Existing Stepsizes and New Bounds}

We now desribe stepsizes previously suggested in the literature.  For simplicity, let $\bM :=\bA^ \top \bA$. Define:
\begin{equation}\label{eq-sigma}
\begin{split}
 \sigma &: = \max \{ x^\top \bM  x : x \in \mathbb{R}^d ; x^\top D^{\bM} x \leq 1 \},\\
\sigma' & :=  \max \{ x^\top \bM x : x \in \mathbb{R}^d ; x^\top B^{\bM}x \leq 1 \} .
\end{split}
\end{equation}
The quantities $\sigma$ and $\sigma'$ are identical to those defined in~\cite{richtarik2013distributed} (although the definitions are slightly different). The following stepsize parameters have been introduced in~\cite{richtarik2013distributed}:
\begin{lemma}[\cite{richtarik2013distributed}]\label{l-oldeso}
The ESO assumption~\ref{ass-eso} is satisfied for
\begin{align} \label{a-Diibeta}
 \bD^2_{ii} \equiv \bD_{ii}=\beta^*\sum_{j=1}^n \bA_{ji}^ 2,\enspace i=1,\dots,d\enspace,
\end{align}
where  $\beta^*:=\beta^*_{1}+\beta^*_{2}$ and
\begin{equation}
\label{eq:sjsus7}
\begin{split}
&\beta^*_{1} = 1+\tfrac{(\tau-1)(\sigma-1)}{s_1},\enspace \beta^*_{2} =\left(\tfrac{\tau}{s} - \tfrac{\tau-1}{s_1}\right) \tfrac{\sigma'-1}{\sigma'}\sigma.
\end{split}
\end{equation}
\end{lemma}

\subsection{Known bounds on existing stepsizes}\label{sec-knownb}

In general, the computation of $\sigma$ and $\sigma'$ can be done using the Power iteration method~\cite{NumericalLA}.
However, the number of operations required in each iteration of the Power method is at least twice as the number of nonzero elements in $\bA$.
 Hence the only computation of the parameters $\sigma$ and $\sigma'$ would already require quite a few number of passes through the data.
Instead, if one could provide some easily computable upper bound for $\sigma$ and $\sigma'$,
where by 'easily computable' we mean computable by only one pass through the data, then we can run Algorithm~\ref{algo:accdistr}
immediately without spending too much time on the computation of $\sigma$ and $\sigma'$.
 Note that the ESO assumption will still hold  when $\sigma$ and $\sigma'$ in~\eqref{a-Diibeta} are replaced by some upper bounds.

In~\cite{richtarik2013distributed}, the authors established the following bound
\begin{align}\label{a-beta2beta1}
\beta^*\leq 2\beta_1^*=2(1+\tfrac{(\tau-1)(\sigma-1)}{s_1})\enspace,
\end{align}
which is independent of the partition $\{\pP_l\}_{l=1,\dots,c}$. This bound holds for $\tau\geq 2$. Further, they showed that
\begin{align}\label{a-upperboundsigma}
 \sigma\leq \max_{j} \omega_j \enspace.
\end{align}
Then in view of~\eqref{a-beta2beta1},~\eqref{a-upperboundsigma} and Lemma~\ref{l-oldeso}, the ESO assumption~\ref{ass-eso} is
also satisfied for the following easily computable parameters:
\begin{align} \label{a-Diibetaupper}
\bD^3_{ii} \equiv \bD_{ii}=2\left(1+\tfrac{\tau-1}{s_1}(\max_j \omega_j-1)\right)\sum_{j=1}^n \bA_{ji}^ 2.
\end{align}

\subsection{Improved bounds on existing stepsizes}

In what follows, we show that both~\eqref{a-beta2beta1} and~\eqref{a-upperboundsigma} can be improved so that
smaller parameters $\{\bD_{ii}\}_{i=1}^d$ are allowed in the algorithm.

\begin{lemma}\label{l-witheta}
Suppose $\tau\geq 2$ (note that then $s\geq \tau\geq 2$).  For all $1\leq \eta \leq s$  the following holds
\begin{equation}\label{eq-teta}
\textstyle \left(\frac{\tau}{s}-\frac{\tau-1}{s-1}\right)\eta \leq
\frac{1}{\tau-1}\left(1+\frac{(\tau-1)(\eta-1)}{s-1}\right)
\end{equation}
\end{lemma}
\begin{proof}
Both sides of the inequality are linear functions of $\eta$. It therefore suffices to verify that the inequality holds for $\eta=1$ and $\eta=s$, which can be done.
\end{proof}

The following result is an improvement on \eqref{a-beta2beta1}, which was shown in \cite[Lemma 2]{richtarik2013distributed}.

\begin{lemma}\label{l-betabeta1}
 If $\tau\geq 2$, then $\beta^* \leq (1+\frac{1}{\tau-1})\beta_{1}^*$.
\end{lemma}
\begin{proof}

We only need to apply  Lemma~\ref{l-witheta} to \eqref{eq:sjsus7} with $\eta = \sigma$, and additionally use the bound $\tfrac{\sigma'-1}{\sigma'}\leq 1$. This gives $\beta_2^*\leq \beta_1^*/(\tau-1)$, from which the result follows.
\end{proof}

The above lemma leads to an important insight:  {\bf as long as $\tau\geq 2$, the effect of  partitioning the data (across the nodes) on the iteration complexity of \DACCORD  is negligible, and vanishes as $\tau$ increases.}

Indeed, consider
the complexity bound provided by stepsizes $\bD^2$ given in \eqref{a-Diibeta}, used within Theorem \ref{thm:main} and note that the effect of the partitioning on the complexity is through $\sigma'$ only (which is not easily computable!), which appears in $\beta_2^*$ only. However, by the above lemma,
$\beta_1^* \leq \beta^* \leq (1+\tfrac{1}{\tau-1})\beta_1^*$ for any partitioning (to $c$ sets each containing $s$ coordinates). Hence, i) each partitioning is near-optimal and ii) one does not even need to know $\sigma'$ to run the method as it is enough to use stepsizes $\bD^2$ with $\beta^*$ replaced by the upper bound provided by Lemma \ref{l-betabeta1}.

Note that the same reasoning can be applied to the Hydra method \cite{richtarik2013distributed} since the same stepsizes can be used there. Also note that for $\tau=1$ the effect of partitioning can have a dramatic effect on the complexity.

\begin{rem}
Applying the same reasoning to the new stepsizes $\alpha_j^*$ we can show that $\alpha_j^ * \leq (1+1/(\tau-1))\alpha_{1,j}^ *$ for all $j\in\{1,\dots,n\}$.
This is not as useful as Lemma~\ref{l-betabeta1} since $\alpha_{2,j}^*$ does not need to be approximated ($\omega_j'$ is easily
computable).
\end{rem}

We next give a tighter upper bound on $\sigma$ than~\eqref{a-upperboundsigma}.
\begin{lemma}\label{l-newupperboundonsigma} If we let
\begin{align}\label{a-vi}
v_i:=  \frac{\sum_{j=1}^ n \omega_j \bA_{ji}^ 2 }{\sum_{j=1}^n \bA_{ji}^ 2},\enspace i=1,\dots,d,
\end{align}
then
 $
\sigma \leq \tilde \sigma := \max_{i} v_i.
$
\end{lemma}
\begin{proof}
In view of Lemma~\ref{l-dso}, we know that
$$
\bM=\sum_{j=1}^ n\bM_j \preceq\sum_{j=1}^ n \omega_j D^ {\bM_j} = \Diag(v)D^{\bM} \preceq \max_i v_i D^ {\bM}.
$$
The rest follows from the definition of $\sigma$.
\end{proof}

By combining Lemma~\ref{l-oldeso}, Lemma~\ref{l-betabeta1} and Lemma~\ref{l-newupperboundonsigma},
we obtain the following smaller (compared to~\eqref{a-Diibetaupper})  admissible and easily computable stepsize parameters:
\begin{equation} \label{a-Diibetauppernew}
 \bD^4_{ii} \equiv \bD_{ii}=\tfrac{\tau}{\tau-1} \left( 1+\tfrac{(\tilde \sigma -1)(\tau-1)}{s-1}\right)\sum_{j=1}^n \bA_{ji}^ 2.
\end{equation}

\section{Comparison of old and new  stepsizes}

So far we have seen four different admissible stepsize parameters, some old and some new.
For ease of reference, let us call $\{\bD^1_{ii}\}_{i=1}^d$ the new parameters defined in~\eqref{a-bDiibetaj}, $\{\bD^2_{ii}\}_{i=1}^d$ the existing one
 given in~\cite{richtarik2013distributed} (see~\eqref{a-Diibeta}),
$\{\bD^3_{ii}\}_{i=1}^d$ the upper bound of $\{\bD^2_{ii}\}_{i=1}^d$ used in~\cite{richtarik2013distributed} (see~\eqref{a-Diibetaupper})
and $\{\bD^4_{ii}\}_{i=1}^d$
 the new upper bound of
$\{\bD^2_{ii}\}_{i=1}^d$ defined in~\eqref{a-Diibetauppernew}.

In the  next result we compare the four stepsizes. The lemma implies that $\bD^1$ and $\bD^2$ are uniformly smaller (i.e., better -- see Theorem \ref{thm:main}) than $\bD^4$ and $\bD^3$. However, $\bD^2$ involves quantities  which are hard to compute (e.g., $\sigma$). Moreover, $\bD^4$ is always smaller than $\bD^3$.
\begin{lemma}
Let $\tau\geq 2$.
 The following holds for all $i$:
\begin{align}
 &\bD_{ii}^ 1 \leq \bD_{ii}^4 \leq \bD_{ii}^3\\
& \bD_{ii}^2\leq \bD_{ii}^4\leq \bD_{ii}^3
\end{align}
\end{lemma}

\noindent {\em Proof.} We only need to show that $\bD_{ii}^1\leq \bD_{ii}^ 4$, the other relations are already proved.
It follows from~\eqref{a-vi} that
\begin{align}\label{a-omegajsigma2}
 \textstyle \sum_{j=1}^n (\omega_j-\tilde \sigma)\bA_{ji}^2 \leq 0 \enspace.
\end{align}
holds for all $i$. Moreover, letting
\begin{equation}\label{eq:sjs8sj}\tilde \sigma'=\max_j \omega_j'\end{equation}
and using Lemma~\ref{l-witheta} with $\eta = \tilde \sigma$, we get
\begin{align}\label{a-tauxtauminus}
(\tfrac{\tau}{s}-\tfrac{\tau-1}{s-1})\tfrac{\tilde\sigma'-1}{\tilde\sigma'}\tilde \sigma \leq \tfrac{1}{\tau-1}( 1+\tfrac{(\tilde \sigma -1)(\tau-1)}{s-1}).
\end{align}
Now, for all $i$, we can write $\bD^{1}_{ii}-\bD^{4}_{ii}=$
\begin{align*}
&\overset{\eqref{a-bDiibetaj}+\eqref{eq-alphaj}+\eqref{a-Diibetauppernew}}{=}\sum_{j=1}^n \bA_{ji}^2\left( \tfrac{\tau-1}{s-1}(\omega_j-\tilde \sigma)\right. +\\& \qquad
\left.(\tfrac{\tau}{s}-\tfrac{\tau-1}{s-1})\tfrac{\omega'_j-1}{\omega'_j}\omega_j-\tfrac{1}{\tau-1}( 1+\tfrac{(\tilde\sigma-1)(\tau-1)}{s-1})\right)\\
&\overset{\eqref{a-omegajsigma2}}{\leq} \sum_{j=1}^n\left((\tfrac{\tau}{s}-\tfrac{\tau-1}{s-1})\tfrac{\omega'_j-1}{\omega'_j}\omega_j-\tfrac{1}{\tau-1}( 1+\tfrac{(\tilde\sigma-1)(\tau-1)}{s-1})\right)\bA_{ji}^2\\
&\overset{\eqref{eq:sjs8sj}+\eqref{a-tauxtauminus}}{\leq}(\tfrac{\tau}{s} -\tfrac{\tau-1}{s_1})\tfrac{\tilde\sigma'-1}{\tilde\sigma'} \sum_{j=1}^ n (\omega_j-\tilde \sigma)\bA_{ji}^2
\overset{\eqref{a-omegajsigma2}}{\leq} 0\enspace. \quad \quad \qed
\end{align*}

Note that we do not have a simple relation between $\{\bD_{ii}^1\}_{i=1}^d$ and $\{\bD_{ii}^2\}_{i=1}^d$. Indeed, for some coordinates $i$ the new stepsize parameters $\{\bD_{ii}^1\}_{i=1}^d$ could be smaller than $\{\bD_{ii}^2\}_{i=1}^d$, see Figure~\ref{fig:svm_daul_evolution2}.
in the next section for an illustration.

  \section{Numerical Experiments}
In this section we present  preliminary computational results involving real and synthetic data sets. In particular,   we first compare the relative benefits of the 4 stepsize rules, and then demonstrate that  \DACCORD can significantly outperform Hydra on a big data non-strongly convex problem involving 50 billion variables.

The experiments were performed on two problems:
\begin{enumerate}
 \item {\bf Dual of SVM:} This problem can be formulated as finding $x \in [0,1]^d$ that  minimizes
\begin{align*}
L(x)=\frac{1}{2\lambda d^2}\sum_{j=1}^n (\sum_{i=1}^d b^i \bA_{ji}x^i)^2 - \frac{1}{d} \sum_{i=1}^d x^i + I(x),
\end{align*} Here $d$ is the number of examples, each row of the matrix $\bA$ corresponds to a feature, $b\in \R^d$, $\lambda>0$ and $I(\cdot)$ denotes the indicator function of the set $[0,1]^d$:
\[
 I(x) = \begin{cases} 0& \text{ if } x\in [0,1]^d, \\ +\infty & \text{ if } x \not \in [0,1]^d.\end{cases}
\]
 \item {\bf LASSO:}
 \begin{equation}L(x)=\frac{1}{2}\norm{\bA x-b}^2+\lambda \norm{x}_1.
 \label{lasso}
 \end{equation}
\end{enumerate}
 The computation were performed on  Archer (\url{http://archer.ac.uk/}),
which is UK's leading supercomputer.

\subsection{Benefit from the new stepsizes}

In this section we compare the four different stepsize parameters and show how they influence the convergence of \DACCORD . In the experiments of this section we used the \emph{astro-ph} dataset.\footnote{Astro-ph  is a binary classification problem which consists of abstracts of papers from physics.
The dataset consists of $d=29,882$ samples and the feature space has dimension $n=99,757$.
Note that for the SVM dual problem, coordinates correspond to samples whereas for the LASSO problem, coordinates correspond to features.} The problem that we solve is the SVM dual problem (for more details, see~\cite{shalev2013stochastic,takavc2013mini}).

Figure \ref{fig:svm_daul_evolution2}
plots the values of $\bD^1_{ii}$, $\bD^2_{ii}$, $\bD^3_{ii}$ and $\bD^4_{ii}$ for $i$ in $\{1,\dots,d\}$
and $(c,\tau)=(32,10)$. The dataset we consider is normalized (the diagonal elements of $\bA^\top\bA$ are all equal to 1), and
the parameters $ \{\bD^2_{ii}\}_i$ are equal to a constant for all $i$  (the same holds for $\{\bD^3_{ii}\}_i$ and $\{\bD^4_{ii}\}_i$).
Recall that according to Theorem~\ref{thm:main},
a faster convergence rate is guaranteed if smaller (but theoretically admissible) stepsize parameters $\{\bD_{ii}\}_{i}$ are used. Hence it is clear that
$\{\bD^1_{ii}\}_i$ and $\{\bD^2_{ii}\}_i$ are better choices than $\{\bD^3_{ii}\}_i$ and $\{\bD^4_{ii}\}_i$. However, as we mentioned,
 computing the parameters
$\{\bD^1_{ii}\}_i$ would require a considerable computational effort, while computing $\{\bD^2_{ii}\}_i$ is as easy as passing once
through the dataset. For this relatively small dataset, the  time is about 1 minute for computing $\{\bD^2_{ii}\}_{i}$ and less than 1 second
for $\{\bD^1_{ii}\}_{i}$.

\begin{figure}[htp]
 \centering
  \includegraphics[scale=0.4]{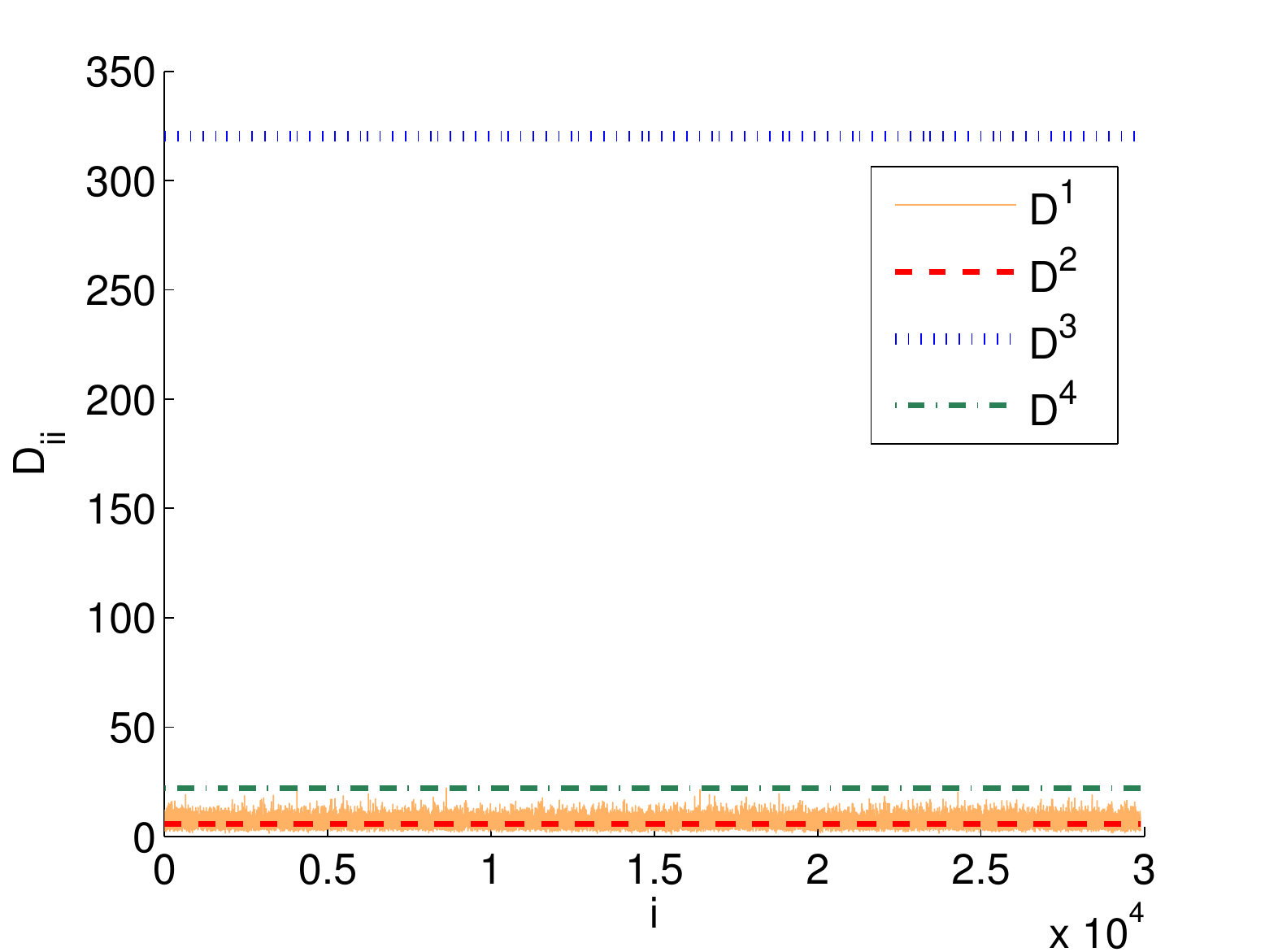}
 \caption{Plots of  $i$ v.s.  $\bD^1_{ii}$,  $i$ v.s. $\bD^2_{ii}$, $i$ v.s. $\bD^3_{ii}$ and $i$ v.s. $\bD^4_{ii}$ }
 \label{fig:svm_daul_evolution2}
\end{figure}

In order to investigate the benefit of the new stepsize parameters, we solved the SVM dual problem
 on the \emph{astro-ph} dataset for $(c,\tau)=(32,10)$. Figure~\ref{fig:svm_daul_evolution} shows evolution of the duality gap, obtained by using the four  stepsize parameters mentioned
previously.
We see clearly that smaller stepsize parameters lead to faster convergence, as  predicted by Theorem~\ref{thm:main}.
Moreover, using our easily computable new stepsize parameters $\{\bD_{ii}^1\}_i$, we achieve
comparable convergence speed with respect to the existing parameters $\{\bD_{ii}^2\}_i$.

\begin{figure}[htp]
 \centering
  \includegraphics[width=5cm]{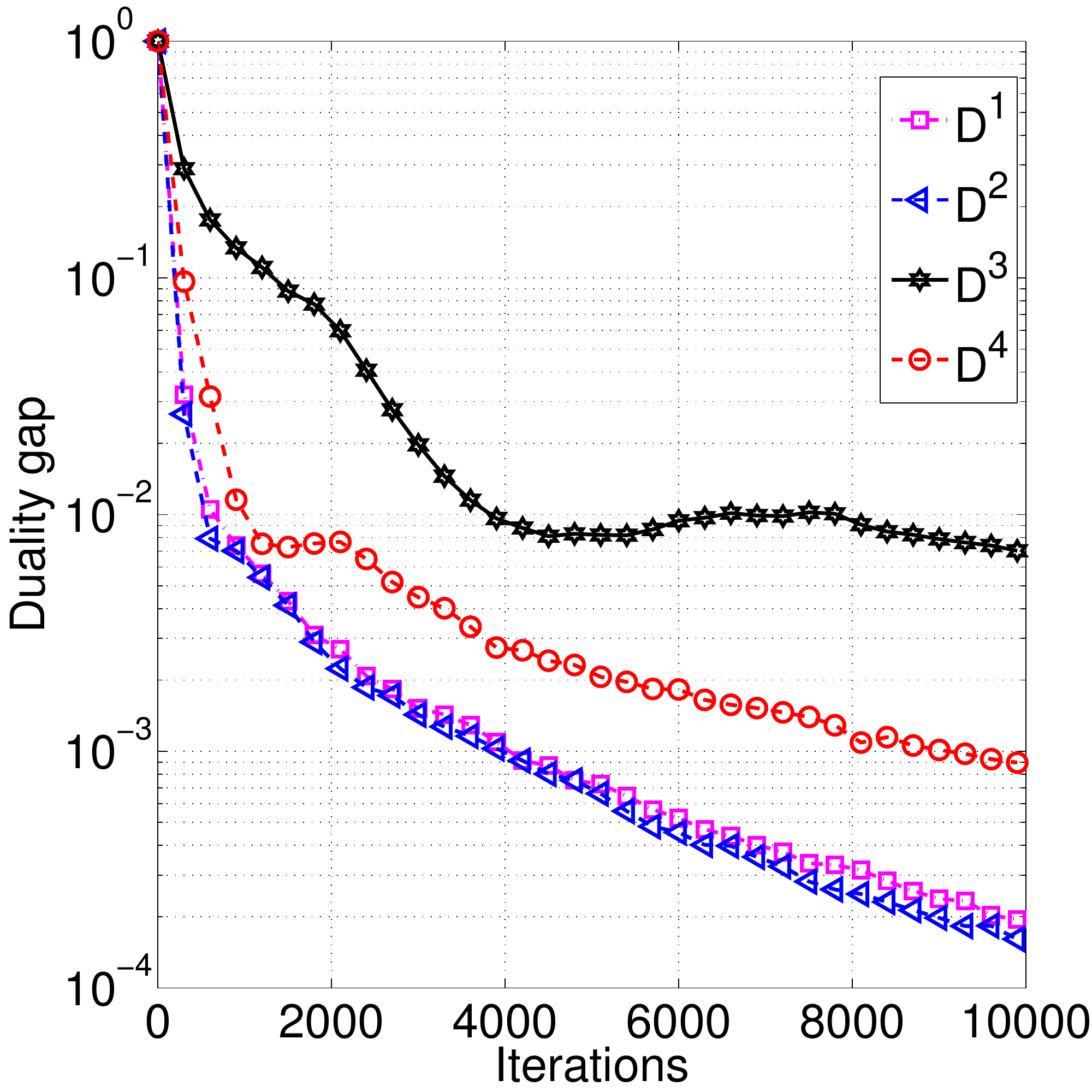}
 \caption{Duality gap v.s. number of iterations, for 4 different stepsize parameters. }
 \label{fig:svm_daul_evolution}
\end{figure}

\subsection{\DACCORD vs Hydra} 
In this section, we report experimental results comparing Hydra with \DACCORD on a synthetic big data LASSO problem. We generated a sparse matrix $\bA$ having the same block angular structure as the one used in~\cite[Sec 7.]{richtarik2013distributed},
with $d = 50$ billion, $c=256$, $s=195,312,500$ (note that $d=cs$) and $n=5,000,000$. 
  The average number of nonzero elements per row of $\bA$ (i.e., $\sum_{j} \omega_j /n$)
 is $60,000$, and the maximal number of nonzero elements in a row of $\bA$ (i.e., $\max_j \omega_j$) is $1,993,419$.  The dataset size is $5$TB.

We have used 128 physical nodes of the ARCHER\footnote{ \url{www.archer.ac.uk}} supercomputing
facility, which is based on Cray XC30 compute nodes connected via Aries interconnect. On each physical node we have run two MPI processes -- one process per NUMA (NonUniform Memory Access) region. Each NUMA region has 2.7 GHz, 12-core E5-2697 v2 (Ivy Bridge) series processor and each core supports 2 hardware threads (Hyperthreads). In order to minimize communication we have chosen $\tau=s/1000$ (hence each thread computed an update for
8,138 coordinates during one iteration, on average).

 We show in Figure~\ref{fig:lassiterNEW} the decrease of the error with respect to the number of
iterations, plotted in log scale. It is clear that \DACCORD provides faster iteration convergence rate than Hydra.
Moreover, we see from  Figure~\ref{fig:lassoNEW} that  the speedup in terms of the number of iterations is significantly large
so that \DACCORD converges faster in time than Hydra even though the run time of \DACCORD per iteration is on average two times more expensive.


\begin{figure}[htb]
 \centering
  \includegraphics[width=5cm]{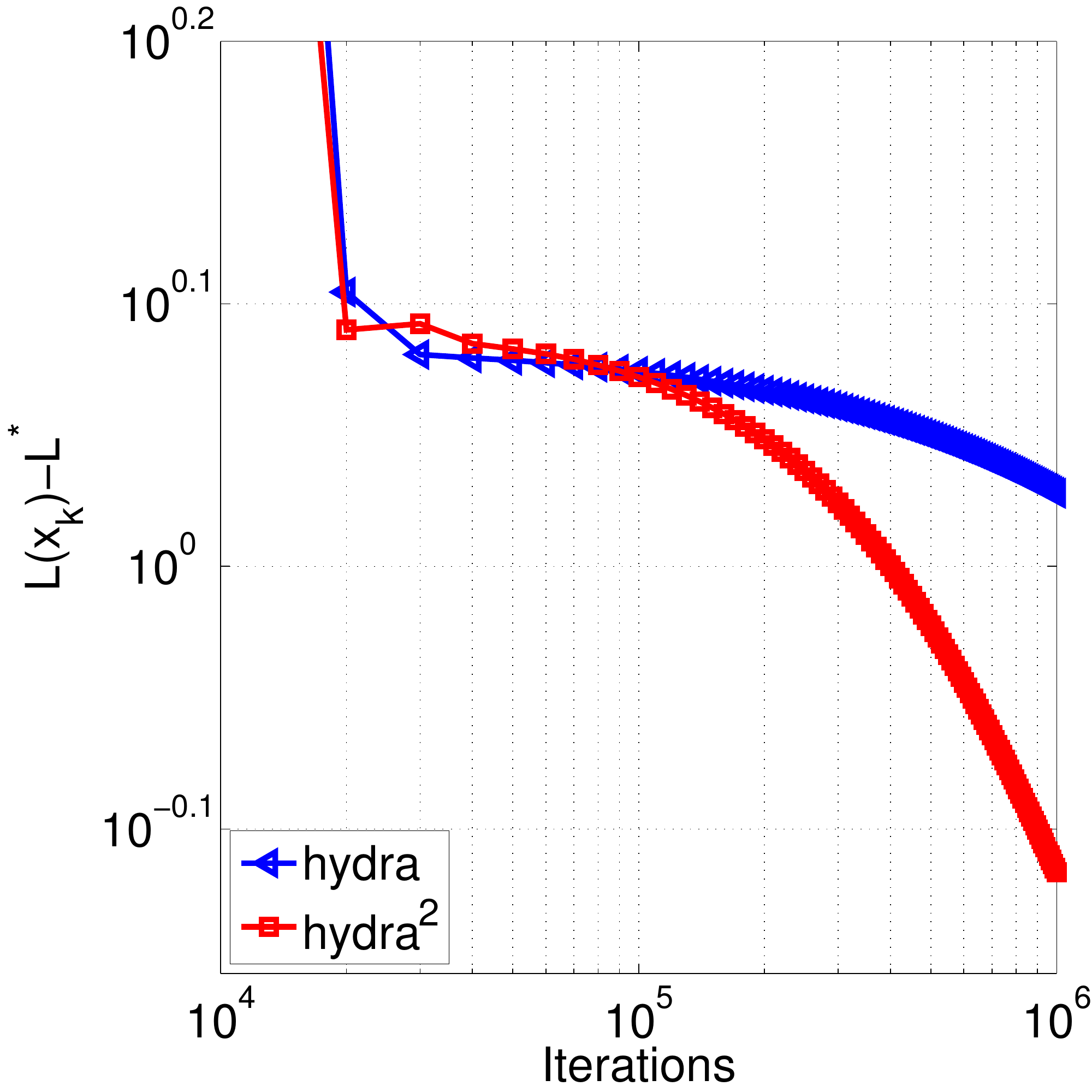}
\caption{Evolution of $L(x_k)-L^*$ in number of iterations.}
\label{fig:lassiterNEW}
\includegraphics[width=5cm]{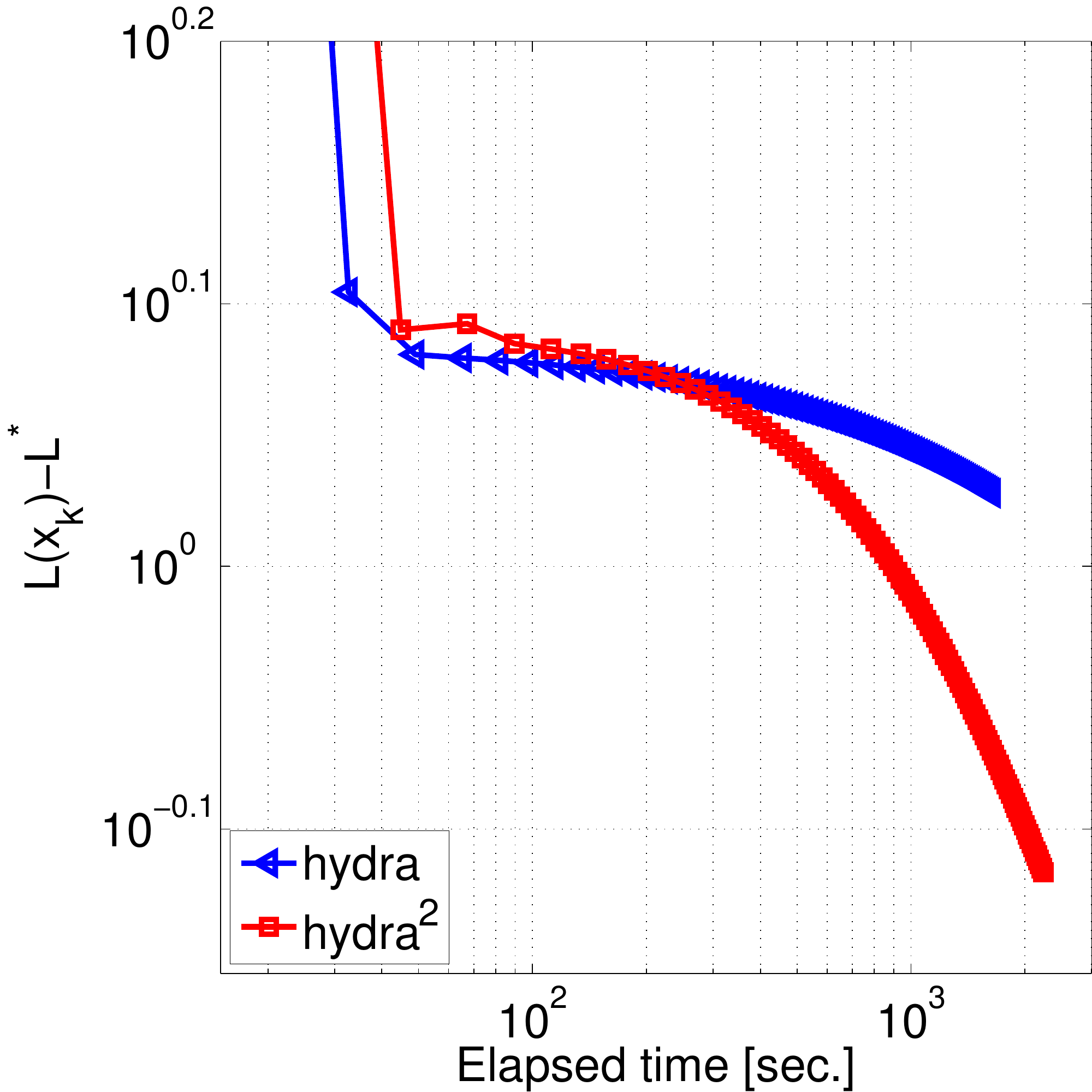}
 \caption{Evolution of $L(x_k)-L^*$ in time.}
 \label{fig:lassoNEW}
\end{figure}

\small
 
\bibliographystyle{IEEEbib}
\bibliography{biblio}

\end{document}